\documentclass[11pt,leqno,twoside]{article}

\usepackage{mathrsfs}
\usepackage{amssymb}
\usepackage{amsthm,amsmath,amscd,amsxtra}
\usepackage{color}

\allowdisplaybreaks

\def\rr{{\mathbb R}}
\def\rn{{{\rr}^n}}

\def\cj{{\mathcal J}}

\def\cl{{\mathcal L}}

\def\cp{{\mathscr P}}

\def\az{\alpha}

\def\lip{{\mathop\mathrm{Lip}}}

\def\bz{\beta}
\def\gz{{\gamma}}

\def\tz{\theta}
\def\sz{\sigma}
\def\wz{\widetilde}

\def\ls{\lesssim}
\def\gs{\gtrsim}

\def\r{\right}
\def\lf{\left}

\newtheorem{thm}{Theorem}[section]
\newtheorem{lem}[thm]{Lemma}
\newtheorem{prop}[thm]{Proposition}

\def\glip{{\mathop\mathrm{GLip}}}

\numberwithin{equation}{section}

\begin{document}

\arraycolsep=1pt

\title{\bf A characterization of the Gaussian Lipschitz space
and sharp estimates for the Ornstein-Uhlenbeck Poisson kernel
\footnotetext{\hspace{-0.22cm}2010
{\it Mathematics Subject Classification}. %
Primary 26A16; Secondary 28C20, 46E35. %
\endgraf {\it Key words and phrases}. %
Gauss measure space, Lipschitz space, Ornstein-Uhlenbeck Poisson kernel.%
\endgraf The first author is supported by
National Natural Science Foundation of China (Grant No. 11101425) and
the Alexander von Humboldt Foundation.
}}
\author{Liguang Liu and Peter Sj\"ogren}
\date{15 April, 2014}
\maketitle

\begin{center}
\begin{minipage}{13.5cm}\small
 \noindent{\bf Abstract.}
The Gaussian Lipschitz space was defined by Gatto and Urbina, by means
of the Ornstein-Uhlenbeck Poisson kernel.
We give a characterization of this space in terms of a combination of ordinary
Lipschitz continuity conditions.
The main tools used in the proof are sharp estimates of the
Ornstein-Uhlenbeck Poisson  kernel and some of its derivatives.
\end{minipage}
\end{center}



\section{Introduction and main results}

\qquad
Let $\gamma$ be the Gauss measure on $\rn$ with $n\ge1$,
that is, $d\gz(x)=\pi^{-n/2}e^{-|x|^2}\,dx$.
The Gaussian analogue of the Euclidean Laplacian is
the \emph{Ornstein-Uhlenbeck operator}
$\cl =-\frac12 \Delta +x\cdot\nabla ,$
where $\nabla:=  (\partial_{x_1}, \dots, \partial_{x_n}).$
The operator $\cl$ is the infinitesimal generator of
the \emph{Ornstein-Uhlenbeck semigroup} $T_t = e^{-t\cl},\;\;t>0$, given by
$$T_tf(x) =\pi^{-n/2}\int_\rn M_{e^{-t}}(x,y)f(y)\,dy$$
 for all $f\in L^2(\gz)$ and $x\in\rn$,
where $M_{e^{-t}}$ is the \emph{Mehler kernel} defined  by
$$M_r(x,y):=\frac{e^{-\frac{|y-rx|^2}{1-r^2}}}{(1-r^2)^{n/2}}
\qquad x,\,y\in\rn,\quad 0<r<1.$$
The \emph{Ornstein-Uhlenbeck Poisson semigroup} $\{P_t\}_{t>0}$  is defined by
subordination from $\{T_t\}_{t>0}$ as
\begin{equation*}
  P_tf(x)=\frac1{\sqrt \pi}\int_0^\infty \frac{e^{-u}}{\sqrt u}
  \,T_{t^2/(4u)}f(x)\,du.
\end{equation*}
There is a corresponding \emph{Ornstein-Uhlenbeck Poisson kernel} $P_t(x,y)$, for
which
$$
P_tf(x)= \int_\rn P_t(x, y) f(y)\,dy,
$$
and it is  obtained from the Mehler kernel
by similar subordination.
 Transforming variables
$s=t^2/(4u)$ and inserting the expression for the Mehler kernel $M_{e^{-s}}$,
one gets
\begin{equation}\label{poisson-ker}
P_t(x,y)= \frac1{2\pi^{(n+1)/2}} \int_0^\infty
\frac{t}{s^{3/2}} \,e^{-\frac{t^2}{4s}}\,
\frac{e^{-\frac{|y-e^{-s}x|^2}{1-e^{-2s}}}}{(1-e^{-2s})^{n/2}} \,ds.
\end{equation}

Gatto and Urbina \cite{GU} introduced
the Gaussian Lipschitz spaces; see also \cite{GPU} and \cite{PU}.
Let $\az\in(0,1)$, which will be fixed throughout the paper.
A function $f$ in $\rn$
is said to be in  the \emph{Gaussian Lipschitz space
$\glip_\az$} if it is
bounded and satisfies
\begin{equation}\label{GLip}
\|\partial_t P_tf\|_{L^\infty}\le A t^{\az-1}, \qquad t>0,
\end{equation}
for some $A>0$.
The norm in $f\in \glip_\az$ is
$$\|f\|_{\glip_\az}:= \|f\|_{L^\infty}+\inf\{A:\, A \,\,\textup{satisfies}\, \,
\eqref{GLip}\}.$$

The standard Euclidean Lipschitz space $\lip_\az(\rn)$
consists of all bounded functions $f$ such that for some $C>0$,
\begin{eqnarray}\label{standard}
|f(x)-f(y)|\le C |x-y|^\az, \qquad x,\,y\in\rn.
\end{eqnarray}
It is known that the space $\lip_\az(\rn)$ can be
characterized by means of the standard Poisson kernel
$$\cp_t(x,y)= c_n \frac{t}{(t^2+|x-y|^2)^{(n+1)/2}};$$
see Stein \cite[Section~V.\,4.\,2]{S1}.
To be precise, a bounded function $f$ belongs to $\lip_\az(\rn)$ if and only if
$$\|\partial_t\cp_t f\|_{L^\infty}\le C t^{\az-1}$$
for all $t>0$. The main aim of this paper is to describe the Gaussian
Lipschitz space by means of a condition like \eqref{standard}, as follows.

\begin{thm}\label{t1.1}
Let $\az\in(0,1)$. The following statements are
equivalent:
\begin{enumerate}
\item[\rm (i)] $f\in \rm{GLip}_\az$;
\item[\rm (ii)] there exists a positive constant $K$ such that
for all $x,y\in\rn$,
\begin{eqnarray}\label{strong-lip-2}
 |f(x)-f(y)|\le K \min\lf\{|x-y|^\az,\;\;
\lf( \frac{|x-y|}{1+|x|+|y|}\r)^{\frac\az 2}+
\lf((|x|+|y|) \sin\theta\r)^\alpha\,
\r\},
\end{eqnarray}
after correction of $f$ on a null set.
Here $\theta$ denotes the angle between the vectors $x$ and $y$;
if $x=0$ or $y=0$, then $\theta$ is understood as $0$.
 \end{enumerate}
Moreover, the norm $\|f\|_{\glip_\az}$ is equivalent to
 $
|f(0)|
 +\inf\{K>0:\, K \,\textup{satisfies}\,   \eqref{strong-lip-2}\}.
$

\end{thm}

In one dimension, the inequality \eqref{strong-lip-2} reads
\begin{eqnarray*}
|f(x)-f(y)|\le K \min\lf\{|x-y|^\az,\;\,
\lf(\frac{|x-y|}{1+|x|+|y|}\r)^{\frac\az2}\,
\r\}.
\end{eqnarray*}
This is a combined Lipschitz condition, with exponent $\az$ for short distance
$|x-y|$ (in fact, shorter than $1/(1+|x|+|y|)$), and exponent $\az/2$,
with a different coefficient, for long
distance. In higher dimension, the expression $(|x|+|y|) \sin\theta$ describes
the ``orthogonal component" of the vector $x-y$, since it is the distance
from $y$ to the line in the direction $x$ plus the vice versa quantity.
To make this more clear, we state an unsymmetric
inequality equivalent to \eqref{strong-lip-2}.
For $x,\,y\in\rn$ with $x\neq0$, we decompose $y$ as $y=y_x+y_x^\prime$,
where $y_x$ is parallel to $x$ and $y_x^\prime $  orthogonal to $x$.
If $x=0$, we let $y_x=y$ and $y_x^\prime=0$,  and this holds for all  $x$  in case  $n
= 1$.
As proved in Lemma \ref{lem-x}
below,
 \eqref{strong-lip-2}  is equivalent to
\begin{equation}\label{strong-lip}
|f(x)-f(y)|\le K' \min\lf\{|x-y|^\az,\;\,
\lf(\frac{|x-y_x|}{1+|x|}\r)^{\frac\az2} +|y_x^\prime|^{\az}\,
\r\}
\end{equation}
 in any dimension, with a constant $K'>0$ comparable with  $K$.
This means that the combined Lipschitz condition applies
in the radial direction,
but in the orthogonal direction the exponent is always $\az$. In the
proof of Theorem~\ref{t1.1}, we shall use \eqref{strong-lip} instead of
\eqref{strong-lip-2}.

The proof of Theorem \ref{t1.1} relies
on pointwise estimates of the Ornstein-Uhlenbeck Poisson kernel
$P_t(x,y)$ and its derivatives,
which also have  independent interest.
Before stating these results, we need some notation.

Throughout the paper, we shall write $C$ for various positive constants which
depend only on $n$
and $\az$.
Given any two nonnegative quantities $A$ and $B$, the notation
$A\lesssim B$ stands for $A\le C B$ (we say that $A$ is controlled by $B$),
and $A\gtrsim B$  means $B\ls A$.
If $B\ls A\ls B$, we write $A\simeq  B$.
For positive quantities $X$,  we shall write
$$\exp^*(-X)$$
meaning  $\exp(-cX)$ for some constant
$c=c(n,\az)>0$ whose value may change from one occurrence to another. Then we have
for
instance $te^{-t}\simeq \exp^*(-t) $ for $t>1$,
since we allow different values
of $c$ in the two inequalities defining the  $\simeq  $ relation.
We shall often use inequalities like $\exp^*(-X) \ls \exp^*(-X)\,\exp^*(-X)$.

\begin{thm}\label{t1.2}
For all  $ t>0$ and $x,\:y\in \rn$,
  \begin{eqnarray*}
 P_t(x, y)  \le C[K_1(t,x,y) + K_2(t,x,y) + K_3(t,x,y) + K_4(t,x,y)],
\end{eqnarray*}
where
 \begin{eqnarray*}
K_1(t,x,y) &=& \frac t{(t^2+|x-y|^2)^{(n+1)/2}}\,\exp^\ast\lf(-t(1+|x|)\r); \\
K_2(t,x,y) &=&\frac t{|x|} \left(t^2+ \frac{|x-y_x|}{|x|}+|y_x^\prime |^2\right)^{
-\frac{n+2}2}\,\exp^\ast\lf(-\frac{(t^2+|y_x'|^2)|x|}{|x-y_x|}\r)\,
\chi_{\{|x|>1,\; x\cdot y>0,\, |x|/2\le |y_x|<|x|\}};\\
 K_3(t,x,y) &=& \mathrm{min}(1,t)\, \exp^\ast(-|y|^2);\\
K_4(t,x,y) &=&\frac t{|y_x|} \left(\log\frac {|x|} {|y_x|}\right)^{ -\frac32}  \,
\exp^\ast\left(-\frac {t^2}{\log\frac  {|x|} {|y_x|}}\right) \,
 \exp^\ast(-|y_x^\prime |^2) \, \chi_{\{x \cdot y>0,\:1<|y_x|<|x|/2\}}.
 \end{eqnarray*}
\end{thm}

In Section \ref{s6}, we consider the sharpness of Theorem \ref{t1.2}.
In particular, we exhibit for each of the four kernels $K_i(t,x,y)$
a set $\wz E_i$
of points $(t,x,y)$ in which $P_t(x,y)\simeq K_i(t,x,y)$
but where the other three terms $K_j(t,x,y)$
are much smaller; see the proof of Theorem \ref{t6.1}(b).
Thus none of the four terms can be suppressed in Theorem \ref{t1.2}.
It can also be verified that  for each $i$ there exist (many) points $(t,x)$
such that the integral of $K_i(t,x,y)$
with respect to $y$, taken over those $y$
for which $(t,x,y) \in\wz E_i $,
is comparable to $1=\int_\rn P_t(x,y)\,dy$.
This means that for these $(t,x)$, the kernel $K_i(t, x,\cdot)$
contains a substantial part of $P_t(x,\cdot)$.

We make some comments about the four terms $K_i$ in
Theorem  \ref{t1.2}, focusing on large values of  $|x|$.

Consider first
small values of $t$.
 The term
$K_1(t,x,y)$ is for $t < 1/(1+|x|)$ essentially the standard Poisson
kernel.
For us, the most significant term is $K_2(t, x,y)$,   since
it is the key to the term with exponent $\alpha/2$ in   \eqref{strong-lip-2}
and \eqref{strong-lip}.
In one dimension,
\begin{eqnarray}\label{1dim}
P_t(x,y) \ls K_2(t,x,y) \ls \wz K_2(t,x,y) : =
\frac1{t^2|x|} \lf(1+\frac{|x-y|}{t^2|x|}\r)^{-\frac32},
\end{eqnarray}
and these estimates are sharp when $x>1$ and $3x/4 < y < x-t^2x$
(see Section 6).
Notice that $\wz K_2(t, x,y)$ is  a Poisson-like kernel
but with a dilation parameter $t^2|x|$ which depends on $x$,
and with a  slower decay as $y\to\infty$.
Further, the integral in  $y$ of each of the three kernels in  \eqref{1dim}
over the interval  $(3x/4, x-t^2x)$ is of order of magnitude
$1 = \int_{\rr} P_t(x,y)\,dy$.
In higher dimension,  $K_2(t, x,y)$ has, as a function of $y$, a
different behavior in the $x$ direction and in the  directions
orthogonal to  $x$.

Our Poisson kernel $P_t$ can be compared with the standard Poisson
kernel $\cp_t$ in the following way. Roughly speaking, the
main part of the standard Poisson integral $\cp_t f(x)$ is essentially the
mean value of the function $f$ in a ball of radius $t$, centered
at $x$.  The analog for $P_t f(x)$ is the mean value in a cylinder in
the $x$ direction of length $t^2|x|$, radius $t$ and center  $x-t^2x$.
This displacement from $x$ of the center is not very significant, since
the displacement is not  larger than the length.

This displacement comes from the Mehler kernel; the subordination
formula  says that  $P_t$ is a weighted mean in the $t$ variable
of values of the Mehler kernel. For small $t$, the Mehler kernel gives
essentially the mean value of the function in
 a ball
of radius $\sqrt t$ and center $e^{-t}x \thickapprox x - tx$. So
for $t<<1/|x|^2$, the displacement is significant here, since it is much
larger than the radius. Actually, it is only this displacement that
makes the Mehler kernel essentially different from the standard heat
kernel, for small $t$.  Observe that the displacement is in the
negative $x$ direction in both cases.

For large $t$, the Mehler kernel has a  dilation factor which is essentially 1,
and the displacement is to the origin. As a result, we get for $P_t$  the terms
  $K_3(t,x,y)$ and  $K_4(t,x,y)$, which are large for small $y$ only.

After finishing this paper, we learned that Garrig\'os et al.\
\cite[Lemmas~4.1 and 4.2]{GHSTV}
  also estimated the kernel  $P_t(x,y)$.
Their estimates are rather different from ours and intended
for other purposes.

From the proof of Theorem \ref{t1.2},
it will be seen that $t\partial_tP_t$
and $t\partial_{x_i}P_t$ with $1\le i\le n$
satisfy the same estimates as $P_t$, as follows.

\begin{thm} \label{t1.3}
Let $i\in\{1,2,\dots,n\}$.
Then
for all $t>0$ and $x,y\in\rn$,
\begin{eqnarray*}
|t\partial_t P_t(x,y)|+|t\partial_{x_i} P_t(x,y)|
\le C \lf[K_1(t,x,y)+K_2(t,x,y)+K_3(t,x,y)+K_4(t,x,y)\r].
\end{eqnarray*}
\end{thm}

For the derivative of $P_t(x,y)$ with respect to $x$ in the radial
direction, i.\,e., along the vector $x$,
we obtain a sharper estimate than that of Theorem \ref{t1.3}.
This result will be of fundamental importance
in the proof of Theorem \ref{t1.1}.
To state it in a simple way, we first observe that
$P_t$ is invariant under rotation in the sense that
$P_t(Ax,Ay)=P_t(x,y)$ for any orthogonal
matrix $A$. The same is true for all the kernels we use.
This means that in our estimates,
we can assume without restriction that
$x=(x_1,0,\dots,0)$ with $x_1\ge0$.
Then we will write the decomposition of $y$ as
$y=(y_1,y')\in\rr\times \rr^{n-1}$.

\begin{thm}\label{t1.4}
For all $t>0$, $x=(x_1,0,\dots,0)\in\rn$ with $x_1\ge0$ and
$y=(y_1,y')\in\rn$,
\begin{eqnarray*}
|\partial_{x_1}P_t(x,y)|
\le C
[Z_1(t,x,y)+Z_2(t,x,y)+Z_3(t,x,y)+Z_4(t,x,y)],
\end{eqnarray*}
where
\begin{eqnarray*}
Z_1(t,x,y)&=&\frac t{(t^2+|x-y|^2)^{(n+2)/2}} \exp^\ast(-t(1+|x|));\\
Z_2(t,x,y)&=&  \frac{t}{x_1^2} \lf(t^2+\frac{x_1-y_1}{x_1}+|y'|^2\r)^{-\frac{n+4}2}
\exp^\ast\lf(-\frac{(t^2+|y'|^2)x_1}{x_1-y_1}\r)\chi_{\{x_1>1, \,x_1/2\le y_1<
x_1\}};\\
Z_3(t,x,y)&=&\frac{\min\{t, \;t^{-2}\}}{1+|x|}\exp^\ast(-|y|^2);\\
Z_4(t,x,y)&=&\frac t{x_1\,y_1} \lf(\log\frac {x_1} {y_1}\r)^{-\frac52}  \,
\exp^\ast\left(-\frac {t^2}{\log\frac  {x_1} {y_1}}\right) \,
 \exp^\ast(-|y'|^2) \, \chi_{\{1<y_1<x_1/2\}}.
\end{eqnarray*}
\end{thm}

The paper is organized as follows. In Section \ref{s2}, we prove the equivalence
between the conditions \eqref{strong-lip-2} and \eqref{strong-lip}
and then give some
basic estimates  needed  later.
Section \ref{s3} contains the proof of Theorem \ref{t1.1}, assuming
Theorems \ref{t1.2}, \ref{t1.3} and \ref{t1.4}.
The proofs of Theorems \ref{t1.2} and \ref{t1.3} are given in Section \ref{s4}.
Section \ref{s5} contains the proof of Theorem \ref{t1.4}, which is based on
that of Theorem \ref{t1.2} but now exploiting also some
cancellation in the integral estimates.
Finally, Section \ref{s6} deals with the sharpness
of our estimates for $P_t$.


\section{Auxiliary results}\label{s2}

\begin{lem}\label{lem-x}
Let  $\alpha \in (0, 1)$. The conditions \eqref{strong-lip-2} and \eqref{strong-lip}
are equivalent,
 and each of them implies that  the function  $f$  is bounded.
More precisely,
\begin{eqnarray}\label{lem-x-e0}
\sup_{x\in\rn} |f(x)-f(0)| \lesssim  \inf K
\simeq  \inf K'.
\end{eqnarray}
\end{lem}

\begin{proof}
 To see that each of the two conditions implies boundedness, it is
 enough to take  $y=0$ in either condition. This also gives the
inequality in \eqref{lem-x-e0}  and the analogous
inequality
 for \eqref{strong-lip}.

Let $A$ and  $B$ denote the minima appearing in
\eqref{strong-lip-2} and \eqref{strong-lip}, respectively.
If $|x|+|y|\le 2$, one finds that $A \simeq |x-y|^\az\simeq  B$.
Assume next that   $|y|/2<|x|<2|y|$. Then  $|y_x^\prime| \simeq (|x|+|y|) \sin\theta$
and  it is obvious that  $B \ls A$. The converse   $A \ls  B$
is easy when  $|y_x^\prime| \le |x-y_x|$. When  $|y_x^\prime| > |x-y_x|$,
we have
\[
 A \le |x-y|^\az \simeq |y_x^\prime|^\az \le B.
\]
Thus it only remains to consider the case when  $|x|+|y|> 2$ and
 $|x|/|y| \notin (1/2, 2)$. But then  $A,\: B \gtrsim 1$, and
via the boundedness we just proved, we see that each of the inequalities
 \eqref{strong-lip-2} and \eqref{strong-lip} implies the other for these
$x,\:y$.

Altogether, this proves the equivalence, and  \eqref{lem-x-e0}
also follows.
\end{proof}

\begin{lem} \label{lem-xx}
Let $a,\, T,\,A\in(0,\infty)$, $X\in[0,\infty)$ and $\bz\in(1,\infty)$. Then,
\begin{eqnarray*}
\cj:=\int_0^a \frac1{\sz^\bz} \exp^\ast\lf(-\frac{T^2}{\sz}\r)
\exp^\ast\lf(-\frac{A^2}{\sz}\r)\,\exp^\ast\lf(-\sz X^2\r)\,d\sz
\le M\,\frac {\exp^\ast\lf(- \frac{AT}{a}\r)\,\exp^\ast\lf(-
T\,X\r)}{(T^2+A^2)^{\bz-1}} ,
\end{eqnarray*}
where $M>0$ is  independent of $a,\, T,\,A$ and $X$.
\end{lem}

\begin{proof}
Notice that $\exp^\ast\lf(-\frac{T^2}{\sz}\r)
\,\exp^\ast\lf(-\sz X^2\r) \ls \exp^\ast\lf(-T\,X\r) $.
Via a change of variable $u= (T^2+A^2)/\sz$, we see that
\begin{eqnarray*}
\cj&&\ls \exp^\ast\lf(-T\,X\r)\,\int_0^a \frac1{\sz^\bz}
\exp^\ast\lf(-\frac{T^2}{\sz}\r)
\exp^\ast\lf(-\frac{A^2}{\sz}\r)\,\,d\sz\\
&&
\ls  \exp^\ast\lf(-T\,X\r)\,\frac{1}{(T^2+A^2)^{\bz-1}}
\int_{\frac {T^2+A^2} a}^\infty u^{\bz-2}\exp^\ast\lf(-u\r)\,du.
\end{eqnarray*}
Since $ (T^2+A^2)/a\ge  2A\,T/{a}$ and $\bz>1$,
the last integral is controlled by
$\exp^\ast\lf(- {AT}/{a}\r)$.
\end{proof}

\begin{prop}\label{p2.1}
For all $x\in\rn$ and $t>0$,  the $K_i$ from  Theorem \ref{t1.2} satisfy
\begin{eqnarray}\label{eq:large}
\int_\rn \lf[K_1(t,x,y)+K_2(t,x,y)\r]\,dy\le C
\end{eqnarray}
and
\begin{eqnarray}\label{eq:small}
\int_\rn \lf[K_3(t,x,y)+K_4(t,x,y)\r]\,dy\le C \min\{1,\,t\}.
\end{eqnarray}
\end{prop}

\begin{proof}
Since $K_1$ is dominated by the standard Poisson kernel, it follows that
$\int_\rn K_1(t,x,y)\,dy\ls1.$ Also, it is obvious that
$\int_\rn K_3(t,x,y)\,dy\ls \min\{1,t\}.$

For the estimates of $K_2$ and $K_4$, we can make a rotation and assume that
$x=(x_1,0,\dots,0)$ with $x_1>1$ and write $y=(y_1,y')$. Then
\begin{eqnarray}\label{p2.1-eq1}
\int_\rn K_2(t,x,y)\,dy
&&\ls  \int_\rr \int_{\rr^{n-1}}\frac t {x_1}
\left(t^2  +  \frac{|x_1-y_1|}{x_1}+  |y'|^2\right)^{ -\frac{n+2}2}\,dy'\,dy_1\\
&&\ls \int_\rr\frac t {x_1}
\left(t^2  +  \frac{|x_1-y_1|}{x_1}\right)^{ -\frac{3}2}\, dy_1 \notag\\
&&\ls 1, \notag
\end{eqnarray}
and \eqref{eq:large} is proved.

In the case of $K_4$, we have
\begin{eqnarray*}
\int_\rn K_4(t,x,y)\,dy
&& \ls  \int_{1<y_1<x_1/2} \int_{\rr^{n-1}}
\frac t{y_1} \left(\log\frac {x_1} {y_1}\right)^{ -\frac32}  \,
\exp^\ast\left(-\frac {t^2}{\log\frac  {x_1} {y_1}}\right) \,
 \exp^\ast(-|y'|^2)\,dy'\,dy_1\\
 &&\simeq  \int_{1<y_1<x_1/2}
\frac t{y_1} \left(\log\frac {x_1} {y_1}\right)^{ -\frac32}  \,
\exp^\ast\left(-\frac {t^2}{\log\frac  {x_1} {y_1}}\right) \,
 \,dy_1\\
 &&\ls \int_{\log 2}^\infty
 \frac t{\sqrt \tau}
\exp^\ast\left(-\frac {t^2}{\tau}\right) \,
 \,\frac{d\tau}{
 \tau}\\
 &&\ls\min\{1,\;t\}.
\end{eqnarray*}
This proves \eqref{eq:small}.
\end{proof}


\section{Proof of Theorem \ref{t1.1}}\label{s3}

\qquad In this section, we assume Theorems \ref{t1.2}, \ref{t1.3}
and \ref{t1.4} and prove Theorem \ref{t1.1}.
Combining Proposition \ref{p2.1}
with the pointwise estimates for the $x$ derivatives of the Poisson kernel
in Theorems \ref{t1.3} and \ref{t1.4},
we first deduce bounds for the $L^1$ norms of those derivatives.

\begin{prop}\label{p3.1x}
\begin{enumerate}
\item[\rm(i)] For all $i\in\{1,2,\dots,n\}$, $t>0$ and $x\in\rn$,
\begin{eqnarray}\label{p2.2-eq1}
\int_\rn |\partial_{x_i} P_t(x,y)|\,dy
\le C t^{-1}.
\end{eqnarray}
\item[\rm (ii)] For all $t>0$ and $x=(x_1,0,\dots,0)\in\rn$ with $x_1\ge0$,
\begin{eqnarray}\label{p2.2-eq2}
\int_\rn |\partial_{x_1} P_t(x,y)|\,dy
\le C  t^{-2}(1+x_1)^{-1}.
\end{eqnarray}
\end{enumerate}
\end{prop}

\begin{proof}
Notice that  (i)
follows from Theorem \ref{t1.3} and Proposition \ref{p2.1}.

To prove (ii), we have,
with $Z_1, Z_2, Z_3, Z_4$  as in Theorem \ref{t1.4},
$$|\partial_{x_1} P_t(x,y)|\ls Z_1(t,x,y)+Z_2(t,x,y)+Z_3(t,x,y)+Z_4(t,x,y).$$
It is easy to see that
\begin{eqnarray*}
\int_\rn  Z_1(t,x,y)\,dy
&&\ls   \int_\rn \frac {t\,\exp^\ast(-t(1+|x|))}{(t^2+|x-y|^2)^{(n+2)/2}}\,dy
\ls t^{-1} \exp^\ast(-t(1+|x|)) \ls  t^{-2}(1+x_1)^{-1}
\end{eqnarray*}
and
\begin{eqnarray*}
\int_\rn  Z_3(t,x,y)\,dy
&&\simeq   \int_\rn \frac{\min\{t, \;t^{-2}\}}{1+|x|}\exp^\ast(-|y|^2)\,dy
\ls \frac{\min\{t, \;t^{-2}\}}{1+|x|} \ls  t^{-2}(1+x_1)^{-1}.
\end{eqnarray*}
Integrating $Z_2$ first  in $y' $ and then in $y_1$, we get
\begin{eqnarray*}
\int_\rn  Z_2(t,x,y)\,dy
&&\ls    \frac{t}{x_1^2} \int_{x_1/2}^{x_1}
 \int_{\rr^{n-1}}
\lf(\frac{x_1}{x_1-y_1}\r)^{(n+4)/2}
\exp^\ast\lf(-\frac{(t^2+|y' |^2)x_1}{x_1-y_1}\r)\,dy'\,dy_1\\
&&\ls \frac{t}{x_1^2} \int_{x_1/2}^{x_1}
\lf(\frac{x_1}{x_1-y_1}\r)^{5/2}
\exp^\ast\lf(-\frac{t^2x_1}{x_1-y_1}\r)\,dy_1\\
&&\ls  t^{-2}x_1^{-1}.
\end{eqnarray*}
Similarly,
\begin{eqnarray*}
\int_\rn  Z_4(t,x,y)\,dy
&&\ls    \frac t{x_1}\int_{0}^{x_1/2}
 \lf(\log\frac {x_1} {y_1}\r)^{-5/2}  \,
\exp^\ast\left(-\frac {t^2}{\log\frac  {x_1} {y_1}}\right) \,
 \,\frac{dy_1}{y_1}\\
&&\ls \frac t{x_1} \int_{\log 2}^\infty u^{-5/2}
\exp^\ast\lf(-\frac{t^2}{u}\r)\,du\\
&&\ls  t^{-2}x_1^{-1},
\end{eqnarray*}
where  $u= \log\frac{x_1}{y_1}$.
Combining these estimates and noticing that $Z_2$ and $Z_4$ are non-zero only if
$x_1>1$, we obtain
\eqref{p2.2-eq2}.
\end{proof}

From this proposition,
we  deduce  two pointwise bounds for the
$x$ derivatives of $P_tf$, with $f$  a Gaussian Lipschitz function.

\begin{prop}\label{p2.3}
Let $\az\in(0,1)$ and $f\in \glip_\az$ with norm $1$.
\begin{enumerate}
\item[\rm (i)] For all $i\in\{1,2,\dots,n\}$, $t>0$ and $x\in\rn$,
\begin{eqnarray}\label{p2.3-eq1}
|\partial_{x_i} P_tf(x)|
\le C t^{\az-1}.
\end{eqnarray}
\item[\rm (ii)] For all $t>0$ and $x=(x_1,0,\dots,0)\in\rn$ with $x_1>0$,
\begin{eqnarray}\label{p2.3-eq2}
|\partial_{x_1} P_tf(x)|
\le C t^{\az-2}(1+x_1)^{-1}.
\end{eqnarray}
\end{enumerate}
\end{prop}

\begin{proof}
To prove (i),
we use the semigroup property of
the Poisson integral and take derivatives,  obtaining
\[
\partial_{x_i}\partial_{t}P_{s+t}f(x)
= \partial_{x_i}\partial_{t} \int_\rn P_s(x,y)P_tf(y)\,dy
=\int_\rn \partial_{x_i}P_s(x,y)\,\partial_{t}P_tf(y)\,dy
\]
for  $s,t>0$ and $x\in\rn$. Now let $s=t$, to get
\[
\frac12 \partial_{x_i}\partial_tP_{2t}f(x)
= \int_\rn \partial_{x_i} P_{t}(x,y)\,
\partial_t P_{t}f(y)\,dy.
\]
By \eqref{p2.2-eq1} and the definition of  $\glip_\alpha$, this
implies that for all $t>0$,
\begin{eqnarray}\label{p2.3-eq3}
\left|\partial_{x_i}\partial_tP_{t}f(x)\right| \lesssim t^{\az-2}.
\end{eqnarray}
Since $f$ is bounded, it follows from \eqref{p2.2-eq1} that
$\partial_{x_i}P_tf(x)\to 0$ as
$t\to \infty$. Thus
\[
\partial_{x_i}P_tf(x)
= -  \int_t^\infty
\partial_{x_i}\partial_\tau P_\tau f(x)\,d\tau,
\]
and (i) is a consequence of this and the preceding inequality.

We prove (ii) by a similar argument, using now \eqref{p2.2-eq2}.
The only difference is that \eqref{p2.3-eq3} is replaced
by
$\left|\partial_{x_1}\partial_tP_{t}f(x)\right| \lesssim t^{\az-3}(1+x_1)^{-1}.$
\end{proof}

\begin{proof}[Proof of Theorem \ref{t1.1}\,]
To prove  that (i) implies (ii), we let $f\in\glip_\az$ with norm $1$
and verify  \eqref{strong-lip}, using Lemma \ref{lem-x}. We start
by modifying $f$ on a null set.
Since $f\in L^\infty(\rn)$ and $\{P_t\}_{t>0}$
is a semigroup to which the Littlewood-Paley-Stein theory applies (see Stein
\cite{S2}), we know that $P_tf(x)\to f(x)$ as $t\to0$ for almost all $x\in\rn$.
For each $t>0$ and all $x\in\rn$, one has
$$P_tf(x)=P_1f(x)-\int_t^1 \partial_\tau P_\tau f(x)\,d\tau,$$
and this integral  has a limit as $t\to 0$ for all $x$.
We  define $f(x)$ as $P_1f(x)-\int_0^1 \partial_\tau P_\tau f(x)\,d\tau$
for all $x\in\rn$.

Let $x,\,y\in\rn$. For any $t>0$, one has
\begin{eqnarray}\label{diff}
|f(x)-f(y)|\le  |f(x)-P_{t}f(x)|+\lf| P_{t}f(x)-P_{t}f(y)\r|+
\lf|P_{t}f(y)-f(y)\r|.
 \end{eqnarray}
Writing the first difference to the right here as an integral
and applying  the definition of $\mathrm{GLip}_\az$, we see that
\begin{eqnarray*}
|f(x)-P_{t}f(x)|
= \lf|\int_0^{t} \partial_\tau P_\tau f(x)\,d\tau\r|
\le \int_0^{t} \tau^{\alpha-1}\,d\tau \simeq t^\az.
 \end{eqnarray*}
The same applies to the third difference.
 For the second difference, Proposition \ref{p2.3}(i) yields that
\begin{eqnarray*}
\lf| P_{t}f(x)-P_{t}f(y)\r|
&&\le |x-y| \sup_{\tz\in(0,1)} |\nabla P_{t} f(x+\tz(y-x))|\ls |x-y|t^{\az-1}.
\end{eqnarray*}
Thus
\begin{eqnarray*}
|f(x)-f(y)|
\ls t^\az+ |x-y|t^{\az-1}.
\end{eqnarray*}
Taking $t=|x-y|$, we get
\begin{eqnarray}\label{thm-char-eq3}
|f(x)-f(y)|\ls |x-y|^\az.
\end{eqnarray}

To verify the remaining part of (ii),
we first make a rotation so that $x=(x_1, 0,\dots,0)$ with $x_1\ge0$.
It is then enough to show that for all $y=(y_1,y')\in\rr\times\rr^{n-1}$,
\begin{eqnarray}\label{aim-1}
|f(x)-f(y)|
\ls \lf({\frac{|x_1-y_1|}{1+x_1}}\r)^\frac{\alpha}2+|y'|^\az.
\end{eqnarray}
Notice that if $|x|=x_1\le2$,
then \eqref{aim-1} follows directly from \eqref{thm-char-eq3},
  since $|f|$ is bounded by $1$.
If $x_1>2$ and $|x_1-y_1|\ge x_1/2$,
the right-hand side of \eqref{aim-1} is  greater than a positive constant,
 so  \eqref{aim-1} follows again.
It only remains to consider the case
$x_1>2$ and $|x_1-y_1|<x_1/2$.
For such $x$ and $y$, we write
$$|f(x)-f(y)|\le |f(x)-f(y_1,0)|+ |f(y_1,0)-f(y)|,$$
and \eqref{thm-char-eq3} already implies that
 $|f(y_1,0)-f(y)|\ls |y'|^\az$.
To estimate $|f(x)-f(y_1,0)|$, we apply \eqref{diff}
 again and proceed as before, but now using \eqref{p2.3-eq2}
 to estimate the $x_1$ derivative. This gives
 that for any $t>0$,
 \[
 |f(x)-f(y_1,0)|
\ls t^\az+ |x_1-y_1| t^{\az-2}\sup_{\tz\in[0,1]} |x_1+\tz(y_1-x_1)|^{-1}.
\]
Since  $|x_1-y_1|<x_1/2$, the supremum here is no larger than $2x_1^{-1}$.
Letting $t=\sqrt{|x_1-y_1|/x_1}$, we obtain
$$
|f(x)-f(y_1,0)|
\ls \lf( \frac{|x_1-y_1|}{x_1}\r)^{\frac\alpha2}\simeq \lf(
\frac{|x_1-y_1|}{1+x_1}\r)^{\frac\alpha2},
$$
so \eqref{aim-1} follows, and (ii)  is verified.

We now prove that (ii) implies (i). Because of Lemma  \ref{lem-x},
we can assume that
\eqref{strong-lip} holds with  $K'\le1$
 and
verify  \eqref{GLip}.
Using  Theorem \ref{t1.3}
and the fact that  $\int_\rn \partial_tP_t(x,y)\,dy=0$, we
can write
\begin{eqnarray*}
|t\partial_tP_tf(x)|
&&= \lf|\int_\rn t \partial_t P_t(x,y)[f(y)-f(x)]\,dy\r|
\le  \sum_{i=1}^4 \int_\rn K_i(t,x,y) |f(y)-f(x)|\,dy.
\end{eqnarray*}
Since the condition \eqref{strong-lip}
implies that $f\in \lip_\az(\rn)$, we have
\begin{eqnarray*}
\int_\rn K_1(t,x,y) |f(y)-f(x)|\,dy
&&  \ls  \int_{\rn}\frac{t}{(t+|x-y|)^{n+1}} |x-y|^\alpha\,dy
\ls t^\alpha.
\end{eqnarray*}
From \eqref{strong-lip}, we deduce that
\begin{eqnarray*}
&&\int_\rn K_2(t,x,y) |f(y)-f(x)|\,dy\\
&&  \quad\ls  \int_\rn\frac t {|x|}
\left(t^2  +  \frac{|x-y_x|}{|x|}+  |y_x^\prime|^2\right)^{ -\frac{n+2}2}
\lf[\lf(\frac{|x-y_x|}{1+|x|}\r)^{\frac\az2} +|y_x^\prime|^\az\r]
\,dy\\
&& \quad \ls  \int_\rn\frac t {|x|}
\left(t^2  +  \frac{|x-y_x|}{|x|}+  |y_x^\prime|^2\right)^{ -\frac{n+2-\az}2}
\,dy.
\end{eqnarray*}
After a rotation of coordinates, we can treat the last integral like the one
in \eqref{p2.1-eq1}; only the exponent is different, and the resulting
bound will be $Ct^\az$.
Finally, Proposition \ref{p2.1} implies  that
\begin{eqnarray*}
\int_\rn \lf[K_3(t,x,y)+K_4(t,x,y)\r] |f(y)-f(x)|\,dy
&&  \ls \|f\|_{L^\infty} \int_\rn \lf[K_3(t,x,y)+K_4(t,x,y)\r] \,dy\\
&& \ls \min\{1,t\}\\
&&\ls t^\alpha.
\end{eqnarray*}
We have verified \eqref{GLip}.
\end{proof}


\section{Proof of Theorems \ref{t1.2} and \ref{t1.3}}\label{s4}

\qquad Since $P_t(x,y)$ and  the $K_i(t,x,y)$ are invariant under
rotation,
 we only need to consider $x=(x_1,0\dots,0)$ with $x_1\ge0$ and  write
$y=(y_1,y')$ as before. Theorem \ref{t1.2} is a consequence of the
slightly sharper result in Proposition \ref{p4.1} below.

 A change of variables $\sz=1-e^{-s}$ in \eqref{poisson-ker} leads to
\begin{equation}\label{poisson-ker-sigma}
P_t(x,y)
= \frac1{2\pi^{(n+1)/2}} \int_0^1 \frac
{t}{s(\sz)^{3/2}}\,e^{-\frac{t^2}{4s(\sz)}}\,
\frac{e^{-\frac{|y-x+\sz x|^2}{1-e^{-2s(\sz)}}}}{(1-e^{-2s(\sz)})^{n/2}}
\,e^{s(\sz)}\,d\sz,
\end{equation}
where  $s(\sz)=\log\frac1{1-\sz}$.
In the sequel, we will split the interval of integration into various subintervals, and
in
each subinterval we use either  $s$ or $\sigma$ as variable of integration.

When $0 < y_1 < x_1$, the quantity
\begin{eqnarray*}
 |y- e^{-s} x|^2 =  |y-x+\sz x|^2=  |y_1-x_1+\sz x_1|^2 +  |y'|^2
\end{eqnarray*}
has a minimum at the point
\begin{eqnarray}\label{sigma0}
\sz_0:=\frac{x_1-y_1}{x_1} \in(0,1),
\end{eqnarray}
and
\begin{eqnarray*}
 |y- e^{-s} x|^2 =(\sz-\sz_0)^2 x_1^2+|y'|^2, \qquad 0<s<+\infty.
\end{eqnarray*}
This  will be used repeatedly in what follows.

\begin{prop}\label{p4.1}
Let $t>0$, $x=(x_1,0,\dots,0)\in\rn$ with $x_1\ge0$ and  $y=(y_1,y')\in\rr\times
\rr^{n-1}$.
\begin{enumerate}
\item[\rm(i)] If $y_1\notin(0,x_1)$, then
\begin{eqnarray}\label{p3.1-eq1}
P_t(x,y) \le C\lf[ K_1(t,x,y)+K_3(t,x,y)\r].
\end{eqnarray}
\item[\rm(ii)] If $y_1\in[x_1/2, x_1)$, then
\begin{eqnarray}\label{p3.1-eq2}
P_t(x,y)  \le C\lf[ K_1(t,x,y)+K_2(t,x,y)+K_3(t,x,y)\r].
\end{eqnarray}
\item[\rm(iii)] If $y_1\in(0,x_1/2)$, then
\begin{eqnarray}\label{p3.1-eq3}
P_t(x,y)  \le C\lf[ K_1(t,x,y)+K_3(t,x,y)+K_4(t,x,y)\r].
\end{eqnarray}
\end{enumerate}

\end{prop}

\begin{proof}
To prove (i), let $y_1\notin(0,x_1)$.
We split the integral in \eqref{poisson-ker-sigma} into
integrals over $(0,1/2)$ and $[1/2,1)$, called $J_1$ and $J_2$.

For  $J_1$, noticing that $\sz\in(0,1/2)$ is equivalent to $s(\sz)\in(0,\ln2)$,
we  have $1-e^{-2s(\sz)}\simeq s(\sz)\simeq \sz$ and $e^{s(\sz)}\simeq 1$.
As a result,
\begin{equation*}
J_1\simeq  \int_0^{1/2}  \frac{t} {\sigma^{(n+3)/2}}
\exp^*{ \left(-\frac{t^2}\sigma\right)}
\exp^*{ \left(-\frac{|y-x+\sigma x|^2}\sigma\right)}\, d\sigma.
\end{equation*}
It follows from $y_1\notin(0,x_1)$ and $\sz<1/2$ that
$|y_1-x_1+\sigma x_1| \gtrsim \max\{\sz x_1,\,|x_1-y_1|\}$,
 and thus
\begin{equation}\label{cond}
|y-x+\sigma x| \gtrsim \max\{\sz |x|,|x-y|\}.
\end{equation}
Notice that for $\sz\in(0,1)$, one has
\begin{equation}\label{cond3}
\exp^*{ \left(-\frac{t^2}\sigma\right)} \ls \exp^*{ \left(-t^2\right)}
\ls \exp^*{ \left(-t\right)}.
\end{equation}
Combined with Lemma \ref{lem-xx}, this yields that
\begin{eqnarray}\label{J1}
J_1&&\ls  \exp^*{ \left(-t\right)}
\int_0^{1/2}  \frac{t} {\sigma^{(n+3)/2}}
\exp^*{ \left(-\frac{t^2+|y-x|^2}\sigma\right)} \,\exp^\ast(-\sz|x|^2)\,
d\sigma\\
&&\ls \frac t{(t^2+|y-x|^2)^{(n+1)/2}}
\exp^\ast\lf(-t(1+|x|)\r)\notag\\
&&\simeq  K_1(t,x,y). \notag
\end{eqnarray}

For $J_2$ we use the variable $s$, getting
\begin{eqnarray}\label{j2-e1}
J_2\simeq \int_{\log 2}^\infty
\frac {t}{s^{3/2}}\,\exp^\ast\lf(-\frac{t^2}{s}\r)\,
\exp^\ast(-|y-e^{-s}x|^2) \,ds.
\end{eqnarray}
Since $y_1\notin (0,x_1)$ and $s\ge\log 2$, one has
$|y-e^{-s}x|\simeq  |y_1-e^{-s}x_1|+|y^\prime |\gs |y_1|+|y^\prime |\simeq
|y|$
and hence
\begin{equation}\label{cond2}
\exp^\ast(-|y-e^{-s}x|^2) \ls \exp^\ast(-|y|^2).
\end{equation}
 Thus,
\begin{eqnarray}\label{j2-e2}
J_2&&\ls  \exp^\ast(-|y|^2)\int_{\log 2}^\infty
\frac {t}{s^{3/2}}\,\exp^\ast\lf(-\frac{t^2}{s}\r)\,
\,ds\ls \min\{1,t\}\exp^\ast(-|y|^2)\simeq K_3(t,x,y).
\end{eqnarray}
We have proved  \eqref{p3.1-eq1} and  (i).

Next, we assume $y_1\in[x_1/2,x_1)$ and prove (ii). With $\sz_0$ given by
\eqref{sigma0} and now satisfying $0<\sz_0\le 1/2$,
we split the integral in \eqref{poisson-ker-sigma} into
integrals over  the three
 intervals $(0,\frac34 \sz_0)$, $[\frac 34\sz_0,\frac54\sz_0]$
and $(\frac54 \sz_0,1)$, denoted $J_{1,1}$, $J_{1,2}$ and $J_{1,3}$,
respectively.

In $J_{1,1}$ we have  $1-e^{-2s(\sz)}\simeq s(\sz)\simeq \sz$
and $e^{2s(\sz)}\simeq 1$, and also
$$|y-x+\sigma x|^2
=(\sz-\sz_0)^2x_1^2+|y'|^2 \simeq \sz_0^2 x_1^2 +|y'|^2 =|x-y|^2.$$
 We get
\begin{equation}\label{j11}
J_{1,1}\simeq  \int_0^{3\sigma_0/4}  \frac{t} {\sigma^{(n+3)/2}}
\exp^*{ \left(-\frac{t^2}\sigma\right)}
\exp^*{ \left(-\frac{|x-y|^2}\sigma\right)}\, d\sigma.
\end{equation}
Since   $|x-y| \gtrsim \sz_0 x_1  \ge \sz x_1 $,
the last $\exp^*$ expression here allows us to introduce also a factor
$\exp^\ast(-\sz|x|^2)$ in the integrand. Because of
 \eqref{cond3}, we can argue as in  \eqref{J1} to get $J_{1,1} \ls K_1(t,x,y)$.

In the integral  $J_{1,2}$,  we have
$\frac{3}{4}\sz_0\le \sz\le \frac{5}{4}\sz_0\le \frac58$
and so  $1-e^{-2s(\sz)}\simeq s(\sz)\simeq \sz\simeq\sz_0$
and $e^{s(\sz)}\simeq1$.  Thus,
\begin{eqnarray}\label{j12}
J_{1,2} && \simeq
 \int_{\frac 34 \sz_0}^{\frac{5}{4}\sz_0}
 \frac{t}{\sigma_0^{(n+3)/2}}
\exp^*{ \left(-\frac{t^2}{\sigma_0}\right)}
\exp^*{ \left(-\frac{(\sz-\sz_0)^2x_1^2+|y'|^2}{\sigma_0}\right)}
\, d\sigma\\
  && \simeq  {t} \lf(\frac{x_1} {x_1-y_1}\r)^{\frac{n+3}2}
\exp^\ast {\left(-\frac{(t^2+|y'  |^2)x_1}{x_1-y_1}\right)}
\int_{\frac34\sz_0}^{\frac54\sz_0}
\exp^*{ \left(-\frac{x_1^3(\sigma -\sz_0)^2}{x_1-y_1}\right)}\, d\sigma,\notag
\end{eqnarray}
where we inserted the expression  \eqref{sigma0} for $\sigma_0$.
The last integral, even extended to the whole line, is
$O((x_1-y_1)/x_1^3)^{ \frac12})$. The $\exp^\ast$ expression preceding it
is now estimated by a product of two factors.  This leads to
\begin{eqnarray*}
J_{1,2}
&&\lesssim  {t} \lf(\frac{x_1} {x_1-y_1}\r)^{\frac{n+3}2}
\min\left\{1, \;
\left( \frac {(t^2+|y'  |^2)x_1}{x_1-y_1}\right)^ {-\frac{n+2}2}\right\}
\exp^\ast {\left(-\frac{(t^2+|y'  |^2)x_1}{x_1-y_1}\right)}
  {\left( \frac{x_1-y_1}{x_1^3}\right)^{ \frac12}}\\
&&  \simeq  K_2(t,x,y)
\end{eqnarray*}
 for $x_1=|x|\ge 1$.

The last integral  in \eqref{j12} is also $O(\sigma_0) = O((x_1-y_1)/x_1)$, and
 we get similarly
\begin{eqnarray}\label{j}
\qquad J_{1,2}
 &&\lesssim   {t} \lf(\frac{x_1} {x_1-y_1}\r)^{\frac{n+1}2}
\,  \min\left\{1, \; \left( \frac {(t^2+|y'|^2)x_1}{x_1-y_1}\right)^
{-\frac{n+1}2}
\right\}\,\exp^\ast {\left(-\frac{t^2x_1}{x_1-y_1}\right)}\,\\
&&\ls t  \min\lf\{\lf(\frac{x_1} {x_1-y_1}\r)^{\frac{n+1}2},\;
\frac 1 {(t^2+|y' |^2)^{(n+1)/2}}
\r\}\, \exp^\ast (-t^2)\notag,
\end{eqnarray}
where we estimated the  $\exp^\ast$ factor by means of the inequality
 ${x_1}/{(x_1-y_1)}>1$. For  $x_1<1$,  one has $x_1-y_1<1$ and so
$(x_1/(x_1-y_1))^{(n+1)/2} \le (x_1-y_1)^{-(n+1)}$, and also
$\exp^\ast (-t^2) \lesssim  \exp^\ast (-t(1+|x|)) $. As a result,
$J_{1,2} \lesssim   K_1(t,x,y) $.

To treat $J_{1,3}$, we  split it
into integrals over the intersection of $(\frac54 \sz_0,1)$ with each of the
 intervals $(0,1/2]$, $(1/2,1)$,
   and denote these  by $J_{1,3}^{(1)}$ and $J_{1,3}^{(2)}$, respectively.

For $J_{1,3}^{(1)}$, we may  assume that
 $\frac54\sz_0<\frac12$; otherwise
 $J_{1,3}^{(1)}=0$.
Since here $\sz\in(\frac 54\sz_0, \frac12]$, we again have
$1-e^{-2s(\sz)}\simeq s(\sz)\simeq\sz$
and $e^{s(\sz)}\simeq1$. Further,
$ (\sz-\sz_0) x_1 \simeq \sz x_1 \ge \sz_0 x_1 = x_1-y_1 $.
Thus $(\sz-\sz_0)x_1  \simeq \mathrm{max}\{\sz x_1, x_1-y_1 \}$,
which implies  \eqref{cond},
 and the argument of  \eqref{J1}  leads to
 $J_{1,3}^{(1)}  \ls K_1(t,x,y).$

Next, we estimate $J_{1,3}^{(2)}$.
For $\max\{\frac54\sz_0,\frac12\}<\sz<1$, we have
$|\sz-\sz_0| x_1 \simeq \sz x_1 \simeq x_1 \simeq y_1$ and
 so
$|y-x+\sz x|^2\gs |y|^2$.
This means that  \eqref{cond2} holds and, arguing as in  \eqref{j2-e2},
we conclude that
 $J_{1,3}^{(2)} \ls K_3(t,x,y)$.
Altogether, we obtain \eqref{p3.1-eq2} and hence (ii).

Finally, we consider (iii), where $y_1\in(0,x_1/2)$ and $\sz_0\in(1/2,\,1)$.
We split the integral in \eqref{poisson-ker-sigma} into
integrals over the
 intervals
 $(0,\sz_0-\frac{y_1}{4x_1})$, $[\sz_0-\frac{y_1}{4x_1},\sz_0+\frac{y_1}{4x_1}]$
and $(\sz_0+\frac{y_1}{4x_1},1)$,
and denote them by $J_{2,1}$, $J_{2,2}$ and $J_{2,3}$, respectively.
Notice that $0<\sz_0-\frac{y_1}{4x_1}<\sz_0+\frac{y_1}{4x_1}<1$.

For $J_{2,1}$, we observe that $\sz<\sz_0-\frac{y_1}{4x_1}$ corresponds to
$s=s(\sz)<\log \frac{x_1}{y_1}-\log \frac 54$.
For such $s$ and $\sz$, one has
$e^{-s}>\frac{5y_1}{4x_1}$ and $|y-x+\sz x|=|y-e^{-s}x|\simeq
e^{-s}x_1+|y'|$.
With $s$ as  variable of integration, we have
\begin{eqnarray*}
J_{2,1}
&&\simeq \int_0^{\log \frac{x_1}{y_1}-\log \frac54}
\frac {t}{s^{3/2}}\,\exp^\ast\lf(-\frac{t^2}{s}\r)\,
\frac{\exp^\ast\lf(-\frac{e^{-2s}x_1^2+|y'  |^2}{1-e^{-2s}}\r)}{(1-e^{-2s})^{n/2}}
\,ds.
\end{eqnarray*}
Splitting the interval of integration here by intersecting it with
 $(0,\log 2]$ and $(\log 2,\infty)$, we obtain two integrals
    denoted   $J_{2,1}^{(1)}$ and $J_{2,1}^{(2)}$.
 For $0<s\le\log 2$, one has $1-e^{-2s}\simeq s\simeq \sigma$  and
$e^{-s}x_1 \simeq x_1\simeq x_1 -y_1 $. This implies  \eqref{cond}
and, arguing as before, we obtain  $J_{2,1}^{(1)}\ls K_1(t,x,y)$.

If $\log 2<s<\log \frac{x_1}{y_1}-\log \frac54$, then $1-e^{-s}\simeq
1$
and  $e^{-s}x_1\gs y_1$, which implies
\eqref{cond2} and then also
 $J_{2,1}^{(2)}\ls K_3$, as before.
We have proved that $J_{2,1}\ls K_1+K_3$.

For  $J_{2,2}$,  we integrate in $s$, getting
\begin{eqnarray*}
J_{2,2}
&&\simeq \int_{\log \frac{x_1}{y_1}-\log \frac54}^{\log
\frac{x_1}{y_1}+\log \frac43}
\frac {t}{s^{3/2}}\,\exp^\ast\lf(-\frac{t^2}{s}\r)\,
\frac{\exp\lf(-\frac{|y-e^{-s}x|^2}{1-e^{-2s}}\r)}{(1-e^{-2s})^{n/2}} \,ds.
\end{eqnarray*}
Since now $x_1>2y_1$, we see that $\log \frac{x_1}{y_1}-\log \frac54
\gs 1$,
which implies that $s\simeq \log \frac{x_1}{y_1}$ and $1-e^{-2s}\simeq 1$
in this integral.
Let $\tau = \log\frac {x_1}{y_1} -s$,
so that $-\log \frac43 \le \tau \le \log \frac54$
and $$|y-e^{-s}x|
\simeq  |y_1-e^{-s} x_1| +|y' |
= |(1-e^{\tau})y_1|+|y'| \simeq  |\tau| y_1+  |y' |.$$
It follows that
\begin{eqnarray}\label{IV}
J_{2,2}
 &&\simeq  \frac t{(\log\frac{x_1}{y_1})^ {3/2}}
\exp^*{ \left(-\frac{t^2}{\log\frac{x_1}{y_1}}\right)} \exp^*{(-|y' |^2)}
\int_{-\log \frac43}^{\log\frac 54}
    \exp^*{(-\tau^2y_1^2)}
  \, d\tau \\
&&\ls    \frac t{(\log\frac{x_1}{y_1})^ {3/2}}
\exp^*{ \left(-\frac{t^2}{\log\frac{x_1}{y_1}}\right)} \exp^*{(-|y' |^2)}
 \frac1{y_1}\notag\\
&&
 \simeq  K_4(t,x,y) \notag
\end{eqnarray}
 when $y_1>1$. If $y_1\in(0,1]$,  we control the integral in
\eqref{IV} by $1$ and  obtain
\begin{eqnarray}\label{eq:4.x1}
\qquad J_{2,2}
&&\simeq  \frac t{(\log\frac{x_1}{y_1})^ {3/2}}
\exp^*{ \left(-\frac{t^2}{\log\frac{x_1}{y_1}}\right)} \exp^*{(-|y' |^2)}
\ls \min\{1,t\} \exp^\ast (-|y'|^2)
 \simeq  K_3(t,x,y).
\end{eqnarray}

In  $J_{2,3}$, we have
$s>\log \frac{x_1}{y_1}+\log \frac43 > \log 2$ and thus
$y_1 - e^{-s}x_1 \simeq y_1$,
which once more leads to \eqref{cond2} and
 $J_{2,3}\ls K_3$.

Summing up,
we obtain \eqref{p3.1-eq3} and (iii).
\end{proof}

%

\begin{proof}[Proof of Theorem \ref{t1.3}]
By differentiating with respect to $t$ in \eqref{poisson-ker}, we have
\begin{equation*}
t\partial_{t} P_t(x,y)
= \frac1{2\pi^{(n+1)/2}} \int_0^\infty \frac {t}{s^{3/2}}\,e^{-\frac{t^2}{4s}}\,
\lf(1-\frac{t^2}{2s}\r)
\frac{e^{-\frac{|y-e^{-s}x|^2}{1-e^{-2s}}}}{(1-e^{-2s})^{n/2}} \,ds.
\end{equation*}
This expression
is similar to that  in \eqref{poisson-ker}, only with an
extra factor $1-t^2/2s$. Since
$$\lf|1-\frac{t^2}{2s}\r|e^{-\frac{t^2}{4s}} \ls \exp^\ast\lf(-\frac{t^2}{s}\r),
$$
we see that
all our  estimates for $P_t$ in Proposition \ref{p4.1}
remain valid for $|t\partial_t P_t|$.

For $i\in\{1,2,\dots,n\}$,
we have
\begin{equation*}
t\partial_{x_i} P_t(x,y)
= \frac1{\pi^{(n+1)/2}} \int_0^\infty \frac {t}{s^{3/2}}\,e^{-\frac{t^2}{4s}}\,
\frac{te^{-s}(y_i-e^{-s}x_i)}{1-e^{-2s}}
\frac{e^{-\frac{|y-e^{-s}x|^2}{1-e^{-2s}}}}{(1-e^{-2s})^{n/2}} \,ds.
\end{equation*}
Compared with  \eqref{poisson-ker},
the integrand here has an extra factor
$$\frac{te^{-s}(y_i-e^{-s}x_i)}{1-e^{-2s}}
=
\frac t{\sqrt s}\; \frac{\sqrt s\,e^{-s}}{\sqrt{1-e^{-2s}}}\;
\frac{y_i-e^{-s}x_i}{\sqrt{1-e^{-2s}}}.
$$
Since the middle factor to the right here is bounded, we can suppress
the extra factor
if we replace $\,e^{-\frac{t^2}{4s}}\,e^{-\frac{|y-e^{-s}x|^2}{1-e^{-2s}}}$ by
$\,\exp^\ast\lf(-\frac{t^2}{s}\r)\exp^\ast(-\frac{|y-e^{-s}x|^2}{1-e^{-2s}})$
in the integral.
Thus
\begin{equation*}
|t\partial_{x_i} P_t(x,y)|
\ls \int_0^\infty \frac {t}{s^{3/2}}\,\exp^\ast\lf(-\frac{t^2}{s}\r)\,
\frac{\exp^\ast(-\frac{|y-e^{-s}x|^2}{1-e^{-2s}})}{(1-e^{-2s})^{n/2}} \,ds,
\end{equation*}
so  the estimates for $P_t$ are valid also for $|t\partial_{x_i}P_t|$.
\end{proof}

\section{Proof of Theorem \ref{t1.4}}\label{s5}

\qquad
Notice that
\begin{equation}\label{x1derivative}
\partial_{x_1} P_t(x,y)
= \frac1{\pi^{(n+1)/2}} \int_0^\infty \frac {t\,e^{-\frac{t^2}{4s}}}{s^{3/2}}\,
\frac{e^{-s}}{\sqrt{1-e^{-2s}}}\,\frac{y_1-e^{-s}x_1}{\sqrt{1-e^{-2s}}}
\frac{e^{-\frac{|y-e^{-s}x|^2}{1-e^{-2s}}}}{(1-e^{-2s})^{n/2}} \,ds.
\end{equation}
We consider the same  cases (i), (ii) (iii) as in Proposition \ref{p4.1},
 and exactly as in the proof of that proposition, we split the integral
into parts by splitting the interval of integration.
The parts will again be denoted by $J_1$, $J_2$, $J_{2,1}^{(1)}$
etc.
For all these parts except  $J_{1,2}$ and  $J_{2,2}$, we
follow closely
the arguments in Section  \ref{s4}; in particular we often use
$\sigma = 1- e^{-s}$ instead of $s$.

Since
\begin{eqnarray*}
\frac{|y-e^{-s}x|}{\sqrt{1-e^{-2s}}} \,
e^{-\frac{|y-e^{-s}x|^2}{1-e^{-2s}}}
\ls \exp^\ast \lf(-\frac{|y-e^{-s}x|^2}{1-e^{-2s}}\r),
\end{eqnarray*}
the absolute value of the integrand in \eqref{x1derivative} is controlled by
\[
 \frac {t\,e^{-\frac{t^2}{4s}}}{s^{3/2}}\,\frac{e^{-s}}{\sqrt{1-e^{-2s}}} \,
\frac {\exp^\ast \lf(-\frac{|y-e^{-s}x|^2}{1-e^{-2s}}\r)}
{(1-e^{-2s})^{n/2}}.
\]
Switching to integration with respect to $\sigma$, we get instead, since
$d\sz =e^{-s} ds $,
\begin{equation*}
 \frac
 {t\,e^{-\frac{t^2}{4s}}}{s^{3/2}}\,\frac1{\sqrt{1-e^{-2s}}} \,
\frac {\exp^\ast \lf(-\frac{|y-x+\sz x|^2}{1-e^{-2s}}\r)}
{(1-e^{-2s})^{n/2}},
\end{equation*}
where $s=s(\sz)=\log 1/(1-\sz)$. Compared with the integral treated in the proof
of  Proposition~\ref{t1.4},
we now have an extra factor which for $s< \log 2$,
i.e., $\sz<1/2$, is controlled by $s^{-1/2} \simeq  \sz^{-1/2}$,
and for  $s > \log 2$ by $e^{-s}$.

For the integrals  $J_1$, $J_{1,1}$, $J_{1,3}^{(1)}$ and $J_{2,1}^{(1)}$, we
integrate in $\sz$ and argue
as in Section \ref{s4}. 
Because of the extra factor   $\sz^{-1/2}$,
the exponent $(n+3)/2$ of  $\sigma$ in \eqref{x1derivative} will now be
 $(n+4)/2$ in the analogous estimates. As a result, the bound
 obtained will  be $Z_1(t,x,y)$ instead of  $K_1(t,x,y)$.

For  $J_2$, $J_{1,3}^{(2)}$ and $J_{2,1}^{(2)}$ and $J_{2,3}$, we use  $s$ as
variable of
integration.  Arguments
similar to those  in  Section \ref{s4}
show that  the integrand is now dominated by
 \begin{equation*}
    \frac {t}{s^{3/2}}\,e^{-\frac{t^2}{4s}}\,e^{-s} \,
 {\exp^\ast (-|y_1-e^{-s}x_1|^2)}
\exp^\ast (-|y|^2).
\end{equation*}
The interval of integration is $(\log 2, +\infty)$ or a subset of it.
Since
\begin{equation*}
\frac {t}{s^{3/2}}\,\exp^\ast\lf(-\frac{t^2}{s}\r) \ls\min\{t, t^{-2}\},
\end{equation*}
the integrals  considered are controlled by
\begin{eqnarray*}
&&\min\{t, t^{-2}\}\,\exp^\ast (-|y|^2)\,
\int_{\log 2}^\infty e^{-s} \,{\exp^\ast (-|y_1-e^{-s}x_1|^2)}\,ds \\
&&\quad\ls
\min\{t, t^{-2}\}\,\exp^\ast (-|y|^2)\,\min\lf\{1, \frac1{x_1}\r\} \\
&& \quad\ls
  Z_3(t,x,y).
\end{eqnarray*}

It remains to estimate  $J_{1,2}$ and  $J_{2,2}$, in which  $y_1$ is as in
 (ii) and
 (iii) of Proposition \ref{p4.1}, respectively.

For  $J_{1,2}$, we thus assume   $y_1\in[x_1/2,\,x_1)$. When
 $0\le x_1 \le 1$, we can estimate
$ |J_{1,2}|$
 as in \eqref{j}. But now the four exponents  $(n+1)/2$ will be replaced by
 $(n+2)/2$, and in the next step, we estimate  $ (x_1/( x_1-y_1))^{(n+2)/2}$
by $( x_1-y_1)^{-(n+2)}$. The result will be
 $ |J_{1,2}| \ls Z_1(t,x,y)$.

When $ x_1 > 1$,
we shall estimate
\begin{eqnarray*}
  J_{1,2}=C \int_{|\sz-\sz_0|\le\frac14\sz_0}
\frac {t}{s(\sz)^{3/2}}\, \exp\lf(-\frac{t^2}{4s(\sz)}\r)\,
\frac{(\sz-\sz_0) x_1}{1-e^{-2s(\sz)}}
\frac{e^{-\frac{|\sz-\sz_0|^2|x|^2+|y'|^2}{1-e^{-2s(\sz)}}}}{(1-e^{-2s(\sz)})^{n/2}}
\,d\sz.
\end{eqnarray*}
Here $\sz_0 \le 1/2$, and $s(\sz) \simeq \sz \le 5/8$ in the integral.
Let us make a change of variable $u= (\sz-\sz_0) x_1$.
Then $\sz=\sz(u)= \sz_0+{u}/{x_1}$,
and we write $s(u)$ for $s(\sz(u))$ so that
\begin{equation}\label{s(u)}
s(u)=\log \frac1{1-\sz(u)}
=\log \frac1{1-\sz_0-{u}/{x_1}}=\log\frac{x_1}{y_1-u}.
\end{equation}
Thus
\begin{eqnarray*}
J_{1,2}
&&=\frac{C}{x_1}\int_{|u|\le\frac{x_1-y_1}{4}}
\frac {t}{s(u)^{3/2}}\, \exp\lf(-\frac{t^2}{4s(u)}\r)\,
\frac{u}{1-e^{-2s(u)}}
\frac{e^{-\frac{u^2+|y'|^2}{1-e^{-2s(u)}}}}{(1-e^{-2s(u)})^{n/2}} \,du\\
&&=\frac{C}{x_1}\int_{|u|\le \frac{x_1-y_1}{4}}
u F(s(u),u)\,du,
\end{eqnarray*}
where for  $\tau\in(0,\infty)$ and $w\in\rr$,
\begin{eqnarray*}
F(\tau,w) = \frac {t}{\tau^{3/2}[1-e^{-2\tau}]^{(n+2)/2}}\,
\exp\lf(-\frac{t^2}{4\tau}\r)\,
\exp\lf(-\frac{w^2+|y'|^2}{1-e^{-2\tau}}\r).
\end{eqnarray*}
Notice that $F(\cdot,w)=F(\cdot,-w)$ for   $w\in\rr$.
We can write
\begin{eqnarray}\label{I12-eq1}
J_{1,2}
&&=\frac{C}{x_1}\int_0^{\frac{x_1-y_1}{4}}
u \lf[F(s(u),u)-F(s(-u),u)\r]\,du,
\end{eqnarray}
and here
\begin{eqnarray}\label{I12-eq11}
|F(s(u),u)-F(s(-u),u)|
\le |s(u)-s(-u)| \sup_{s(-u)<\tau<s(u)} \lf|{\partial_\tau F(\tau,u)}\r|.
\end{eqnarray}
From  \eqref{s(u)} and the mean value theorem, we deduce that
for $0 < u \le (x_1-y_1)/4 $
\begin{eqnarray}\label{I12-eq2}
|s(u)-s(-u)|
&&\le 2u \sup_{-u<v<u}\frac1{y_1-v}.
\end{eqnarray}
With  $s(-u)<\tau<s(u)$, we have
\begin{eqnarray*}
\partial_\tau F(\tau,u)
= F(\tau,u)
\lf[ -\frac3{2\tau} - \frac{(n+2)e^{-2\tau}} {1-e^{-2\tau}}
+\frac{t^2}{4\tau^2}
+\frac{2(u^2+|y'|^2)e^{-2\tau} }{(1-e^{-2\tau})^2}\r].
\end{eqnarray*}
Here ${(n+2)e^{-2\tau}}/(1-e^{-2\tau})\ls \tau^{-1}$, and
$\frac{t^2}{4\tau^2} \exp\lf(-\frac{t^2}{4\tau}\r)
\ls \tau^{-1}\exp^\ast(-\frac{t^2}\tau)$. Further,
\begin{eqnarray*}
\frac{(u^2+|y'|^2)e^{-2\tau} }{(1-e^{-2\tau})^2}\,
\exp\lf(-\frac{u^2+|y'|^2}{1-e^{-2\tau}}\r)
&&\ls
\frac{e^{-2\tau}}{1-e^{-2\tau}}\,\exp^\ast\lf(-\frac{u^2+|y'|^2}{1-e^{-2\tau}}\r)\\
&&\ls \frac1\tau \, \exp^\ast\lf(-\frac{u^2+|y'|^2}{1-e^{-2\tau}}\r),
\end{eqnarray*}
and so
\begin{eqnarray}\label{I12-eq3}
\lf|{\partial_\tau F(\tau,u)}\r|
&&\ls  \frac{t}{\tau^{5/2}[1-e^{-2\tau}]^{(n+2)/2}}\,
\exp^\ast\lf(-\frac{t^2}{\tau}\r)
\exp^\ast\lf(-\frac{u^2+|y'|^2}{1-e^{-2\tau}}\r).
\end{eqnarray}

Recall that  $x_1/2\le y_1<x_1$. In \eqref{I12-eq2} we have
  $|v|<u\le (x_1-y_1)/{4}<{y_1}/2$
so  that $y_1-v\simeq y_1\simeq x_1$, and we conclude that
$$|s(u)-s(-u)| \ls \frac{u}{x_1}.$$
Since all occurring values of $s(\pm u)$ and $\tau$ satisfy
 $s(\pm u) \simeq \tau \simeq\sz_0$, \eqref{I12-eq3}  implies
\begin{eqnarray*}
\sup_{s(-u)<\tau<s(u)} \lf|{\partial_\tau F(\tau,u)}\r|
&&\ls  \frac{t}{\sz_0^{(n+7)/2}}
\exp^\ast\lf(-\frac{t^2}{\sz_0}\r)
\exp^\ast\lf(-\frac{u^2+|y' |^2}{\sz_0}\r).\notag
\end{eqnarray*}
Inserting the last two estimates in \eqref{I12-eq11}, we obtain
$$
|F(s(u),u)-F(s(-u),u)|
\ls \frac{u}{x_1}\,\frac{t}{\sz_0^{(n+7)/2}}\,
\exp^\ast\lf(-\frac{t^2}{\sz_0}\r)\,
\exp^\ast\lf(-\frac{u^2+|y'  |^2}{\sz_0}\r),
$$
which combined with \eqref{I12-eq1} implies that
\begin{eqnarray*}
|J_{1,2}|
&&\ls  \frac t{x_1^2\, \sz_0^{(n+7)/2}}\,\exp^\ast\lf(-\frac{t^2}{\sz_0}\r)
\int_{0}^{\frac{x_1-y_1}{4}}u^2\,
\exp^\ast\lf(-\frac{u^2+|y'|^2}{\sz_0}\r)\,du\notag\\
&&
\ls \frac t{x_1^2} \, \frac 1{\sz_0^{(n+4)/2}}\,
\exp^\ast\lf(-\frac{t^2 + |y'|^2}{\sz_0}\r) \notag\\
&&\lesssim \frac t{x_1^2} \, \frac 1{\sz_0^{(n+4)/2}}\,
\min \lf\{1, \lf(\frac{\sz_0} {t^2 + |y'|^2}\r)^{(n+4)/2}  \r\}\,
\exp^\ast\lf(-\frac{t^2 + |y'|^2}{\sz_0}\r).
\end{eqnarray*}
Since $\sigma_0 = (x_1-y_1)/x_1$, we see that the last expression amounts
to  $Z_2(t,x,y)$.

We shall finally estimate $J_{2,2}$, in which $y_1\in(0, x_1/2)$.
 When $0< y_1\le 1$, we  have an
upper estimate for  $|J_{2,2}|$
 like \eqref{eq:4.x1}, but now
with an extra factor $\exp(-\log  \frac{x_1}{y_1}) \ls (1+x_1)^{-1}$
coming  from $e^{-s}$; recall that $|s -\frac{x_1}{y_1} | \ls 1$. Thus
\begin{eqnarray*}
|J_{2,2}|
\ls  \frac t{(\log\frac{x_1}{y_1})^ {3/2}}\,
\exp^*{ \left(-\frac{t^2}{\log\frac{x_1}{y_1}}\right)}\,
\frac1{1+x_1}\, \exp^*{(-|y' |^2)}.
\end{eqnarray*}
The first $\exp^*$ factor here is controlled by
$\min\left\{1,\lf(\frac{t^2}{\log (x_1/y_1)}\r)^{-3/2} \right\}$. Since $\log
\frac{x_1}{y_1} \gtrsim 1$,
this is seen to lead to $|J_{2,2}| \ls Z_3(t,x,y) $.

When $y_1 > 1$, we estimate  $J_{2,2}$
 by modifying the preceding argument for $J_{1,2}$.
Instead of \eqref{I12-eq1}, we get now
\begin{eqnarray}\label{I22-eq1}
J_{2,2}
&&=\frac{C}{x_1}\int_0^{\frac{y_1}{4}}
u \lf[F(s(u),u)-F(s(-u),u)\r]\,du,
\end{eqnarray}
and we still  have  \eqref{s(u)},  \eqref{I12-eq11}, \eqref{I12-eq2} and
\eqref{I12-eq3}.
 Since now $0 \le u \le {y_1}/{4}$ in \eqref{I12-eq2},  it follows that
$y_1-v\simeq y_1$ for any $|v|<u$,
and thus
$$
|s(u)-s(-u)| \ls \frac{u}{y_1}.
$$
In the estimate for $J_{2,2}$ in  Section \ref{s4},
we saw that
$s\simeq \log \frac{x_1}{y_1}$, which now means that
$s(u)\simeq s(-u)\simeq\log \frac{x_1}{y_1}$, and $\log \frac{x_1}{y_1} > \log 2$.
 In \eqref{I12-eq3},
we thus have $\tau \simeq \log \frac{x_1}{y_1}$ so that $1-e^{-2\tau}\simeq 1$,
which implies that
\begin{eqnarray*}
\sup_{s(-u)<\tau<s(u)} \lf|\partial_\tau F(\tau,u)\r|
&&\ls t\, {\lf(\log \frac{x_1}{y_1}\r)^{-5/2}}
\exp^\ast\lf(-\frac{t^2}{\log \frac{x_1}{y_1}}\r)\,
\exp^\ast\lf(-u^2-|y^\prime |^2\r).
\end{eqnarray*}
Inserting the last two estimate in \eqref{I12-eq11}, we see that
$$
|F(s(u),u)-F(s(-u),u)|
\ls \frac{t\,u}{y_1}\,{\lf(\log \frac{x_1}{y_1}\r)^{-5/2}}
\exp^\ast\lf(-\frac{t^2}{\log \frac{x_1}{y_1}}\r)\,
\exp^\ast\lf(-u^2-|y'|^2\r),
$$
which combined with \eqref{I22-eq1} implies that
\begin{eqnarray*}
|J_{2,2}|
&&\ls \frac t{x_1\,y_1}
{\lf(\log \frac{x_1}{y_1}\r)^{-5/2}}
\exp^\ast\lf(-\frac{t^2}{\log \frac{x_1}{y_1}}\r)
\exp^\ast\lf(-|y^\prime |^2\r)
\int_{0}^{\frac{y_1}{4}}
u^2
\exp^\ast\lf(-u^2\r)\,du\\
&&
\ls \frac t{x_1\,y_1}
{\lf(\log \frac{x_1}{y_1}\r)^{-5/2}}
\exp^\ast\lf(-\frac{t^2}{\log \frac{x_1}{y_1}}\r)
\exp^\ast\lf(-|y' |^2\r) \notag\\
&&\simeq Z_4(t,x,y).
\end{eqnarray*}

 Theorem \ref{t1.4} is proved.
\hfill$\square$


\section{Sharpness arguments} \label{s6}

We let $K_j(t,x,y), \; j=1,2,3,4$, be as in Theorem \ref{t1.2}.
Let $\rr_+=(0,\infty)$.

\begin{thm}\label{t6.1}
\begin{enumerate}
\item[\rm(a)] The  estimate $ P_t(x, y) \simeq  K_1(t,x,y)$ holds uniformly in
    the set
\begin{eqnarray*}
E_1=&&
\bigg\{(t,x,y)\in \rr_+\times \rr^n\times \rr^n:\\
&&  |x|>1,\;\; x\cdot y>0, \; \;
  t^2|x|<|x|-|y_x| < \frac1{4|x|},\;\; |y_x^\prime |< |x|-|y_x|\,
   \bigg\}.
\end{eqnarray*}
Similarly,  $ P_t(x, y) \simeq  K_2(t,x,y)$ uniformly in
\begin{eqnarray*}
E_{2} =&&
\Bigg\{(t,x,y)\in \rr_+\times \rr^n\times \rr^n:\\
&&  |x|>1,\;\; x\cdot y>0,\;\;  t|x|>1,\; \;
t^2|x|<|x|-|y_x| < |x|/4,\;\; |y_x^\prime |< \sqrt{\frac{|x|-|y_x|}{|x|}}\,
\Bigg\},
\end{eqnarray*}
and  $ P_t(x, y) \simeq  K_3(t,x,y)$ uniformly in
\[
E_{3} =
\{(t,x,y)\in \rr_+\times \rn\times \rn :\;\; t>1,\;\; |x|<1, \; \; |y|<1
\,\}.
\]
Finally,  $ P_t(x, y) \simeq   K_4(t,x,y)$ uniformly in
\begin{eqnarray*}
E_{4} =&&
\Big\{(t,x,y)\in \rr_+\times \rn\times \rn :\\ && |x|>e^{16},\; \;
t= \frac{\sqrt{\log|x|}}2 ,  \;\;
|x|^{2/3} \le |y_x| \le |x|^{3/4}, \;\; |y_x^\prime |<1\Big\}.
\end{eqnarray*}
\item[\rm(b)] In the estimate in Theorem \ref{t1.2}, none of the terms
$K_i(t, x,y), \,\: i=1,2,3,4,$
can be suppressed.
\end{enumerate}
\end{thm}

\begin{proof}
To prove (a),
we  only need to consider
$x=(x_1,0,\dots,0)$ with $x_1\ge0$ and write $y=(y_1,y')$.
We shall use several estimates from the proof of Proposition
\ref{p4.1}. Observe that points of  $E_1$ and
  $E_{2}$ belong to  (ii) of Proposition
\ref{p4.1} and satisfy $t<1/2$.

Assume
 $(t,x,y)\in E_1$.  Then
 $$x_1>1,\;\; \; t^2 x_1<x_1-y_1 < x_1^{-1}/4   
 \quad \textup{ and }\quad |y^\prime |< x_1-y_1.$$
Transforming variables in  the integral in \eqref{j11}, we get
\[
J_{1,1} \simeq    \frac t{(t^2+|x-y|^2)^{(n+1)/2}}
\int_0^{B}
\frac1{u^{(n+3)/2}}
\exp^*{\left(-\frac1u\right)}\, du,   
\]
with $B = 3(x_1-y_1)/(4x_1(t^2+|x-y|^2)) $.
One easily verifies that $B^{-1}\ls 1$, so that the value of the integral here
stays away from 0. Since also  $t(1+|x|)\lesssim 1$, it follows
 that $ J_{1,1}\simeq   t/{(t^2+|x-y|^2)^{(n+1)/2}} \simeq  K_1(t,x,y)$.
 Consequently,
 $P_t(x,y)\gs K_1(t,x,y)$ in  $E_1.$

To obtain the converse inequality,
we notice that Proposition \ref{p4.1}(ii) applies, and its proof shows
that $P_t(x,y) \ls
J_{1,1}+ J_{1,2} +  J_{1,3} \ls  K_1(t,x,y) +
J_{1,2} + K_3(t,x,y) $. The inequalities  \eqref{j} now imply that
$J_{1,2} \ls  K_1(t,x,y)$,
since  $x_1/(x_1-y_1) < (x_1-y_1)^{-2} $ in $E_1$. Further,
$$K_3(t,x,y)\simeq t\exp^*{(-|y|^2)} \lesssim t\exp^*{(-|x|^2)}
\lesssim  K_1(t,x,y).$$
We conclude that  $ P_t(x, y) \simeq  K_1(t,x,y)$  in  $E_1$.

Now assume $(t,x,y)\in E_{2}$ so that
$$
x_1>1,\; \; \; tx_1>1,\;\; \;
t^2x_1<x_1-y_1 < x_1/4 \quad \textup{ and }\quad |y^\prime |<
\sqrt{{(x_1-y_1)}/{x_1}}.
$$
Then
 $K_2(t,x,y) \simeq  t x_1^{n/2} (x_1-y_1)^{-(n+2)/2}.$
Since
$x_1(x_1-y_1)>t^2x_1^2>1$,
 a simple scaling shows that  the second
integral in   \eqref{j12} has order of magnitude
$((x_1-y_1)/x_1^3)^{1/2}$.
The  exp* factor preceding it
is essentially $1$, and we conclude  that
$$
 J_{1,2} \simeq  {t} \lf(\frac{x_1} {x_1-y_1}\r)^{\frac{n+3}2}
\lf(\frac {x_1-y_1}{x_1^3}\r)^{\frac{1}2} \simeq  K_2(t,x,y).
$$
Thus $ P_t(x, y) \gtrsim  K_2(t,x,y)$.
In $E_{2}$ one also has
 $K_1(t,x,y)  \lesssim t/(x_1-y_1)^{n+1}  $ and
$K_3(t,x,y)  \lesssim t\exp^*{(-x_1^2)} $, and these quantities are controlled
by  $K_2(t,x,y).$
 Proposition  \ref{p4.1}(ii) then shows that $ P_t(x, y) \lesssim  K_2(t,x,y)$.
Thus $ P_t(x, y) \simeq  K_2(t,x,y)$  in  $E_{2}$.

 Assume next  that   $(t,x,y)\in E_{3}$ so that  $ K_3(t,x,y) \simeq  1$.
 Now \eqref{j2-e2} is sharp and leads to
$J_2 \simeq  1\simeq  K_3$.  Also,   $K_2(t, x,y) = K_4(t,x,y)=0 $, and
$K_1(t,x,y) \ls t^{-n} \ls  1$.
It follows that   $ P_t(x, y) \simeq  K_3(t,x,y)$  in  $E_{3}$.

Finally let $(t,x,y)\in E_{4}$.
Then the estimate  \eqref{IV} is  sharp since $y_1>1$, and so
$J_{2,2} \simeq  K_4(t,x,y)$. Further, one verifies that
$ K_4(t,x,y) \gs  x_1^{-3/4}(\log  x_1)^{-1} $ and also
 that $K_1(t,x,y)$ and $ K_3(t,x,y)$
 are controlled by $\exp^*{(-x_1)}\ls K_4(t,x,y) $.
It now follows from Proposition \ref{p4.1}(iii) that
$ P_t(x, y) \simeq  K_4(t,x,y)$  in  $E_{4}$.

This completes the arguments for  (a).

We prove (b) by finding for each  $\varepsilon > 0$  and  $i = 1,
2,3,4$  a nonempty subset  $\tilde E_i$  of  $E_i$  in which
$K_j < \varepsilon  P_t$ for  $j \ne i $.
In the proof below, we fix $\varepsilon$ and denote by $C_\varepsilon$
various large positive
 constants which may depend on $\varepsilon$.

Let
\begin{eqnarray*}
\wz E_1=
\left\{(t,x,y)\in E_1:
\;  |x|>C_\varepsilon,\; \;   t=\frac1{|x|^2},\;\;  \frac1{|x|^2}<|x|-|y_x| <
\frac2{|x|^2}
  \right\}.
\end{eqnarray*}
In this set, $P_t(t,x,y) \simeq K_1(t,x,y)\simeq  |x|^{2n}$ but
$K_2(t,x,y) \simeq  |x|^{3n/2}$ and
$K_3(t,x,y)\lesssim 1$, whereas  $K_4(t,x,y)$ vanishes.
A suitable choice of  $C_\varepsilon$ yields the desired inequalities.

In a similar way, we define
 \begin{eqnarray*}
\wz E_{2} =
\left\{(t,x,y)\in E_2:\,
  |x|>C_\varepsilon,\; \;  t=|x|^{-1/2},\;\; 1<|x|-|y_x| <2 \right\},
\end{eqnarray*}
 and it is enough to observe that in this set
 $P_t(t,x,y) \simeq K_2(t,x,y) \simeq  |x|^{(n-1)/2} $,
 but  $K_1(t,x,y)\ls \exp^*{(-|x|^{1/2})}  $
 and   $ K_3(t,x,y) \lesssim\exp^*{(-|x|^2)} $
 and $ K_4(t,x,y) = 0$.

The next set is
 \[
\wz E_{3} =
\left\{(t,x,y)\in E_3:\; t>C_\varepsilon\right\},
\]
in which $P_t(t,x,y) \simeq K_3(t,x,y) \simeq  1 $
but   $K_1(t,x,y)\lesssim t^{-n}$
 and  $K_2(t, x,y) = K_4(t,x,y) = 0$.

Finally,
\begin{eqnarray*}
\wz E_{4} =
\{(t,x,y)\in E_4 :\;  |x|>C_\varepsilon\}.
\end{eqnarray*}
To compare the kernels $K_i(t,x,y)$ on this set, it is enough to consider the
last part of the proof of (a).

This ends the proof of (b) and that of the theorem.
\end{proof}

\medskip

\noindent {\sc Liguang Liu\,}

\noindent  Department of Mathematics,
School of Information\\
Renmin University of China\\
Beijing 100872\\
China

\noindent {\it E-mail}: \texttt{liuliguang@ruc.edu.cn}

\bigskip

\noindent {\sc Peter Sj\"ogren\,}

\noindent  Mathematical Sciences,
University of Gothenburg\\ and\\
Mathematical Sciences, Chalmers\\
SE-412 96  G\"oteborg\\
Sweden

\noindent {\it E-mail}: \texttt{peters@chalmers.se}

\end{document}